\documentclass[11pt]{article}

\usepackage{amsmath,amssymb,amsthm,setspace,graphicx}
\usepackage[text={6.5in,8.4in}, top=1.3in]{geometry}

\newtheorem{theorem}{Theorem}[section]
\newtheorem{proposition}[theorem]{Proposition}
\newtheorem{lemma}[theorem]{Lemma}

\begin{document}

\title{Paths with many shortcuts in tournaments \thanks{This research was supported by the Israel Science Foundation (grant No. 1082/16).}}

\author{
Raphael Yuster
\thanks{Department of Mathematics, University of Haifa, Haifa
31905, Israel. Email: raphy@math.haifa.ac.il}
}

\date{}

\maketitle

\setcounter{page}{1}

\begin{abstract}

A {\em shortcut} of a directed path $v_1 v_2 \cdots v_n$ is an edge $v_iv_j$ with $j > i+1$.
If $j = i+2$ the shortcut is a {\em hop}. If all hops are present, the path is called {\em hop complete} so the path and its hops form a {\em square} of a path.
We prove that every tournament with $n \ge 4$ vertices has a Hamiltonian path with at least $(4n-10)/7$ hops,
and has a hop complete path of order at least $n^{0.295}$.

A spanning binary tree of a tournament is a {\em spanning shortcut tree} if for every vertex of the tree, all its left descendants are in-neighbors and all its right descendants are out-neighbors.
It is well-known that every tournament contains a spanning shortcut tree.
The number of shortcuts of a shortcut tree is the number of shortcuts of its unique induced Hamiltonian path.
Let $t(n)$ denote the largest integer such that every tournament with $n$ vertices has a spanning shortcut tree with at least $t(n)$ shortcuts. We almost determine the asymptotic growth of $t(n)$ as it is proved that
$\Theta(n\log^2n) \ge t(n)-\frac{1}{2}\binom{n}{2} \ge \Theta(n \log n)$.

\vspace*{3mm}

\end{abstract}

\section{Introduction}

All graphs in this paper are finite and simple. An {\em orientation} of an undirected graph is obtained by assigning a direction to each
edge. An important class of oriented graphs are {\em tournaments} which are
orientations of a complete graph. A classical result of R\'edei \cite{redei-1934} states that every tournament has an odd number of Hamiltonian paths, and in particular at least one such path. In this paper we establish the existence of Hamiltonian paths or other long paths in tournaments that are ``rich'' in the precise sense
that follows. For a large collection of results on Hamilton paths of given type in tournaments we refer to
the recent survey \cite{BH-2018}.

A {\em shortcut} of a directed path $v_1 v_2 \cdots v_n$ is an edge $v_iv_j$ with $j > i+1$.
If $j = i+2$ the shortcut is a {\em hop}. If all hops are present, the path is called {\em hop complete} so the path and its hops form a {\em square} of a path.
More generally, the $k$th power of a directed path as above contains all edges $v_iv_j$ where
$i+1 \le j \le i+k$. One can naturally define the following three parameters with respect to $n$-vertex tournaments.
Let $h(n)$ denote the largest integer such that every
tournament with $n$ vertices has a Hamiltonian path with at least $h(n)$ hops.
Let $s(n)$ denote the largest integer such that every
tournament with $n$ vertices has a Hamiltonian path with at least $s(n)$ shortcuts.
Let $c(n)$ denote the largest integer such that every
tournament with $n$ vertices has a hop-complete path on $c(n)$ vertices.

Our first result concerns $h(n)$. It is not difficult to construct tournaments for which every Hamiltonian path has at most $\lceil (2n-6)/3 \rceil$ hops (see Section 2).
While it is plausible to suspect that the construction is extremal for all $n$, we can
obtain a relatively close lower bound.
\begin{theorem}\label{t:1}
	For all $n \ge 4$ it holds that $\lceil (2n-6)/3 \rceil \ge h(n) \ge (4n-10)/7$.
\end{theorem}
Our proof technique can be used to slightly improve the lower bound ratio $4/7$ at the expense of introducing considerably more technical details, but it cannot reach the $2/3$ upper bound ratio.

We next turn to consider the largest hop-complete path in a tournament.
It is relatively simple to prove that a tournament all of whose vertices have out-degree roughly $n/2$ has
a hop-complete Hamiltonian path. In fact, Bollob\'as and H\"aggkvist \cite{BH-1990} proved a significantly stronger result: for every $\epsilon > 0$, if $n$ is sufficiently large, then every
tournament with $n$ vertices having the property that the out-degree and in-degree of each vertex is at least
$(\frac{1}{4}+\epsilon)n$ has a $k$th power of a Hamiltonian {\em cycle}. The constant $1/4$ cannot be decreased as for any smaller constant it is easy to construct tournaments
satisfying the resulting requirement for the out-degree and in-degree and that are not strongly connected. One may then ask whether long hop-complete paths can still be guaranteed in {\em every} tournament, or, more formally, to determine $c(n)$.
It is very easy to obtain a logarithmic lower bound for $c(n)$ just by using the fact that an $n$-vertex tournament has a transitive sub-tournament of logarithmic order. Here we prove that $c(n)$ is
at least polynomial in $n$.
\begin{theorem}\label{t:2}
	$c(n) \ge n^{0.295}$.
\end{theorem}
While we cannot rule out that $c(n)$ is linear in $n$, we can show that there are tournaments
of order $n$ where the longest $k$th power of a path is of order about $nk/2^{k/2}$ (see Proposition \ref{p:1}).

The value of $s(n)$ easily relates to another well-studied graph parameter.
Let $\beta(n)$ be the largest integer such that every tournament with $n$ vertices has
an acyclic subgraph with at least $\beta(n)$ edges.
While it is straightforward that $\beta(n) \ge  \frac{1}{2}\binom{n}{2}$, determining its growth beyond this
lower bound is not trivial. Spencer \cite{spencer-1971}, improving earlier results of
Erd\H{o}s and Moon \cite{EM-1965}, proved that $\beta(n) \ge \frac{1}{2}\binom{n}{2} + \Omega(n^{3/2})$.
A probabilistic construction of Spencer \cite{spencer-1980}, later simplified
with an improved constant by de la Vega \cite{delavega-1983} shows that
$\beta(n) \le \frac{1}{2}\binom{n}{2} + O(n^{3/2})$, hence the growth rate of $\beta(n)$ above the
trivial threshold is $\Theta(n^{3/2})$.
It is a folklore argument to show that every maximal acyclic subgraph of a tournament has a Hamiltonian path, hence it immediately follows that $s(n)=\beta(n)-n+1$
and that $s(n) = \frac{1}{2}\binom{n}{2} + \Theta(n^{3/2})$.
It is not difficult to prove that there is also an $O(n^2)$ time algorithm that produces
a maximal acyclic subgraph of a tournament with $\frac{1}{2}\binom{n}{2} + \Omega(n^{3/2})$ edges.

There are two standard, equally simple proofs that every tournament has a Hamiltonian path.
The first is the greedy construction which extends every non-Hamiltonian path by adding a non-path vertex to it.
The second is a recursive construction: Take any vertex $v$ of a tournament, construct (recursively) two Hamiltonian paths on the sub-tournaments induced by the in-neighbors and out-neighbors of $v$ respectively. Now concatenate these paths together with $v$ in the middle to a Hamiltonian path of the entire tournament. This latter Hamiltonian path has an interesting tree-like structure which we can formally define as follows.

Recall that a binary tree is a rooted tree where every vertex has at most two children,
a left child (if exists) and a right child (if exists). A descendant of a rooted tree vertex
is any other vertex that appears in the subtree rooted at that vertex.
In a binary tree, left (resp. right) descendants are all vertices that appear in the subtree rooted at that left
(resp. right) child.
A spanning binary tree of a tournament is a {\em spanning shortcut tree} if for every vertex of the tree, all its left descendants are in-neighbors and all its right descendants are out-neighbors, as shown in Figure \ref{f:shortcut-tree}.
Observe that every spanning shortcut tree is associated with a unique Hamiltonian path formed by the in-order traversal of
the tree (the in-order is the unique order of the vertices where for each vertex, all its left descendants appear before it in the order, and all its right descendants appear after it in the order). The number of shortcuts of the spanning shortcut tree is defined as the number of shortcuts of its Hamiltonian path. Note that the recursive construction of a Hamiltonian path described in the previous paragraph actually
constructs a spanning shortcut tree of the given tournament.
In fact, it is easy to see that the number of shortcuts of a spanning shortcut tree is
$\Omega(n \log n)$ and Bar-Noy and Naor \cite{BN-1990} designed an $O(n \log n)$ time algorithm
that produces a spanning shortcut tree and in particular, a Hamiltonian path in that running time.

\begin{figure}
	\includegraphics[scale=0.6,trim=-190 310 600 45, clip]{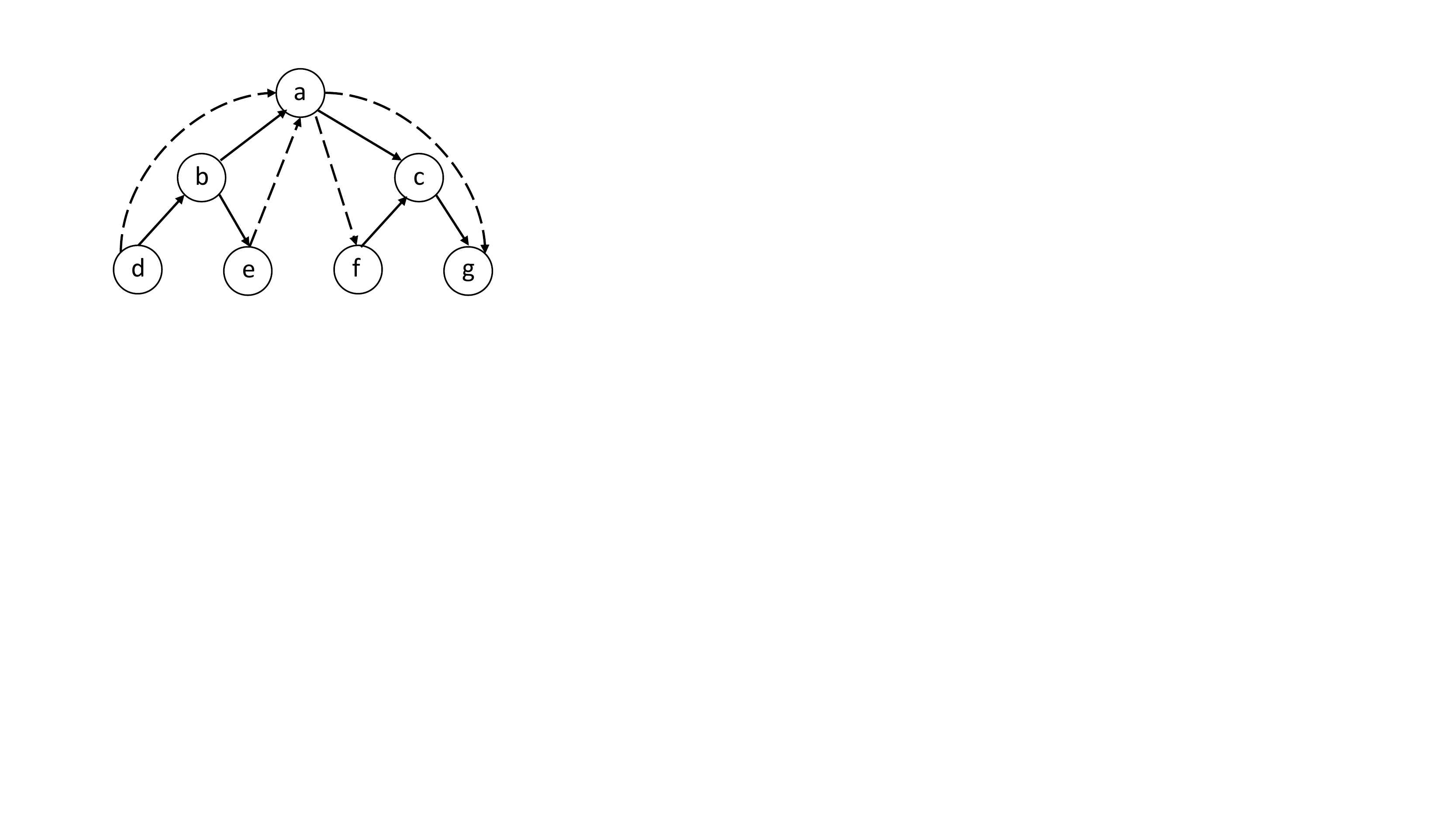}
	\caption{A shortcut tree of some tournament on vertex set $\{a,b,c,d,e,f,g\}$. The solid arrows are tree edges (edges connecting a vertex with its child)
	and the dashed arrows represent edges connecting a vertex to its non-child descendant. For every vertex $v$, edges go from $v$ to its right descendants, and go into $v$ from its left descendants. The in-order of the tree is
	$d,b,e,a,f,c,g$ and it is a Hamiltonian path. The dashed arrows are shortcuts but the tree may have other shortcuts; for example if $df$ is an edge then it is a shortcut, while if $fd$ is an edge then it is not a shortcut.}
	\label{f:shortcut-tree}
\end{figure}

Notice that not every Hamiltonian path of a tournament is a Hamiltonian path of a spanning shortcut
tree. Furthermore, it is easy to construct tournaments and a Hamiltonian path in them that does correspond to a spanning shortcut tree but only has $\Theta(n \log n)$ shortcuts.
As we are interested in  ``rich'' Hamiltonian paths,
let $t(n)$ be the largest integer such that every tournament with $n$ vertices has a spanning shortcut tree with at least $t(n)$ shortcuts.
Our next result proves that the growth rate of $t(n)$ above the $50$ percent range is $\Omega(n \log n)$
(Note: it is not entirely obvious why even $50$ percent is always attainable). We also get quite close to this lower bound showing that $t(n)$ is not larger than $O(n \log^{2} n)$ above the $50$ percent range.
\begin{theorem}\label{t:3}
For all $n \ge 1$ it holds that $t(n) \ge \frac{1}{2}\binom{n}{2} +(n\ln n)/5 - 3122$.
Furthermore, there is an $O(n^2)$ time algorithm that finds a spanning shortcut
tree with at least $\frac{1}{2}\binom{n}{2} +(n\ln n)/5 - 3122$ shortcuts.
On the other hand, for all $n$ sufficiently large, $t(n) \le \frac{1}{2}\binom{n}{2} +4n\log^{2} n$.
\end{theorem}

The rest of this paper consists of sections in which we prove the aforementioned theorems and some additional
results. Throughout the paper we only use standard graph theoretic terminology.
Some of the notations we use frequently are $d^+(v)$ and $d^-(v)$ to denote the out-degree
and in-degree of a vertex $v$, $N^+(v)$ and $N^-(v)$ to denote the set of out-neighbors and
the set of in-neighbors, and $G[X]$ to denote the sub-tournament of a tournament $G$ induced
by a vertex set $X$. An edge from $u$ to $v$ is denoted by $uv$.
The unique acyclic (thereby transitive) tournament with $n$ vertices is denoted by $T_n$.
Finally, $\ln n$ denotes the natural logarithm and $\log n$ denotes the base $2$ logarithm.

\section{Hops}

We first show the construction yielding the upper bound of Theorem \ref{t:1}.
Consider the case where $n$ is a multiple of $3$.
Take $n/3$ pairwise vertex-disjoint directed triangles denoted by $X_1,\ldots,X_{n/3}$.
Now, if $i < j$, orient all $9$ edges connecting $X_i$ and $X_j$ from $X_i$ to $X_j$.
Denote the resulting tournament by $R_n$. In every Hamiltonian path of $R_n$, all the vertices of $X_i$
precede all the vertices of $X_{i+1}$ for $i=1,\ldots,n/3-1$. Hence all three vertices of $X_i$ appear
consecutively, and since $X_i$ is a directed triangle, there is no hop in the sub-path induced by $X_i$.
So, the total number of hops is precisely $n-2-n/3=(2n-6)/3$. If $n \equiv 1 \pmod 3$ then define
$R_n$ by taking $R_{n-1}$ and adding a vertex of in-degree $n-1$. 
If $n \equiv 2 \pmod 3$ then define
$R_n$ by taking $R_{n-2}$ and adding a vertex of in-degree $n-2$ and then a vertex of in-degree $n-1$.
As the number of hops of every Hamiltonian path of $R_n$ is precisely $\lceil (2n-6)/3 \rceil$,
the upper bound of Theorem \ref{t:1} follows.

In the following lemmas and proofs we use the notation $h(G)$ for the maximum number of shortcuts in a
Hamiltonian path of a tournament $G$.
\begin{lemma}\label{l:h7}
	For all $n \ge 7$ it holds that
	$$
	h(n) \ge \min_{\frac{1}{4}n-\frac{1}{2} \le x \le \frac{3}{4}n-\frac{1}{2}} h(x)+h(n-1-x)+2\;.
	$$
	If $G$ is a tournament with $n$ vertices and with a source or a sink then $h(G) \ge h(n-1)+1$ and if it has a vertex with out-degree or in-degree $1$ then $h(G) \ge h(n-2)+1$.
\end{lemma}
\begin{proof}
	Let $G$ be a tournament with $n \ge 7$ vertices.
	It is well-known, and easy to verify that $G$ has a vertex $v$ with
	$\frac{1}{4}n-\frac{1}{2} \le d^+(v) \le \frac{3}{4}n-\frac{1}{2}$.
	Let $v$ be such a vertex and let $x=d^+(v)$. Notice that $x \ge 2$ and $n-x-1 \ge 2$.
	Consider the tournaments $G_1=G[N^+(v)]$ and $G_2=G[N^{-}(v)]$ where $G_1$ has $x$ vertices and $G_2$ has $n-x-1$ vertices.
	Let $P_1$ be a Hamiltonian path of $G_1$ with at least $h(x)$ hops and
	let $P_2$ be a Hamiltonian path of $G_2$ with at least $h(n-x-1)$ hops.
	Consider now the Hamiltonian path of $G$ formed by concatenating $P_2,v,P_1$.
	Then it has at least $h(x)+h(n-1-x)+2$ hops since $x$ and the second vertex of $P_1$ form a hop
	and also $x$ and the second to last vertex of $P_2$ form a hop.
	The second part of the lemma is proved in a similar manner.
\end{proof}

\begin{lemma}\label{l:h10}
	For all $2 \le n \le 10$, we have $h(n)=\lceil (2n-6)/3 \rceil$ and furthermore, for $n=3,6$ the unique extremal tournament is $R_n$.
\end{lemma}
\begin{proof}
	The lemma trivially holds for $n=2,3$ and observe that $h(3)=0$ is only obtained by $C_3=R_3$.
	It is also easy to verify that $h(4)=1$. For $n=5$, consider a $5$-vertex tournament $G$. If $G$ has a vertex of out-degree $0$ or a vertex of out-degree $4$ then $h(G) \ge 2$ follows from $h(4)=1$. Otherwise, $G$ must have a vertex $v$ of out-degree $2$
	(since some out-degree in a $5$-vertex tournament must be even). Then a Hamiltonian path where the first two vertices are the in-neighbors of $v$ and the last two vertices are the out-neighbors of $v$ shows that $h(G) \ge 2$. As we also have $h(R_5)=2$ we obtain $h(5)=2$.
	Consider next a tournament $G$ on $6$ vertices. If $G$ has a vertex $v$ with
	$2 \le d^+(v) \le 3$ then assume without loss of generality that $d^+(v)=3$ (as otherwise $d^-(v)=3$ and
	the proof is analogous). Suppose $x,y,z$ are the out-neighbors of $v$ and that w.l.o.g. $xy,yz$ are edges.
	Suppose that $u,w$ are the in-neighbors of $v$ and that w.l.o.g. $uw$ is an edge. If $xz$ is an edge
	then $u,w,v,x,y,z$ is a Hamiltonian path with three hops. Otherwise, $zx$ is an edge and $x,y,z$ induce a $C_3$. If $w$ has an out-neighbor in $\{x,y,z\}$ then assume without loss of generality that it is
	$x$, then $u,w,v,x,y,z$ is a Hamiltonian path with three hops. Otherwise, each of $x,y,z$ is an in-neighbor of $w$. If $u$ has an in-neighbor in $\{x,y,z\}$ then assume without loss of generality that it is $x$, then $x,u,w,v,y,z$ is a Hamiltonian path with three hops. Otherwise, $u,v,x,y,z,w$ 
	is a Hamiltonian path with three hops. Assume next that all the out-degrees of $G$ are either $0,1,4,5$.
	So three of them must be $1$ and three of them must be $4$ and $G$ must therefore be $R_6$.
	
	Consider next a tournament $G$ on $7$ vertices. If $G$ has a vertex of out-degree in $\{0,1,2,4,5,6\}$ then $h(G) \ge 3$ using the recursive construction in the proof of Lemma \ref{l:h7}. Otherwise, $G$ is a regular tournament. If $G$ has a vertex whose out-neighbors induce a $T_3$ then the recursive construction in the proof of Lemma  \ref{l:h7} also gives
	$h(G) \ge 3$. Otherwise, $G$ must be the Paley tournament and in fact $h(G)=4$ in this case.
	We therefore have $h(7)=3$.
	
	Consider next a tournament $G$ on $8$ vertices. If there is an out-decree in $\{0,2,5,7\}$ we have
	$h(G) \ge 4$ by the recursive construction in the proof of Lemma \ref{l:h7}. Otherwise, all out-degrees are in $\{1,3,4,6\}$. If there is a vertex $v$ with out-degree $6$ (or, analogously, out-degree $1$) then either $h(G) \ge 4$ or else by the
	extremity of $R_6$, the out-neighbors of $v$ induce $R_6$. Furthermore, if $u$ is the unique in-neighbor of $v$ then all three vertices of $X_1$ (the first triangle of $R_6$) are in-neighbors of $u$ (as otherwise we have $h(G) \ge 4$). But if this is the case, then the out-degree of each vertex of $X_1$ is
	$5$ which we assume is not the case. So, we remain with the case where all the out-degrees of $G$ are
	in $\{3,4\}$ and we wish to prove that $h(G) \ge 4$ in this case as well. Assume otherwise, that
	$h(G)=3$. Assume without loss of generality that $v$ has out-degree $4$ (otherwise it has in-degree $4$ and the proof is analogous). If we have $h(G)=3$ then we must have that $G_2=G[N^{-}(v)]$ has
	$h(G_2)=0$  so it induces a $C_3$, and we must have that $G_1=G[N^{+}(v)]$ has $h(G_1)=1$
	so it has a source or a sink.  If it has a source, then that source is the first vertex of any Hamilton
	path of $G_1$, so all the vertices of $G_2$ are out-neighbors of the source, but then the out-degree of the source in $G$ is $6$, contradiction. If it has a sink then the other three vertices of $G_1$
	call them $x,y,z$ must induce a $C_3$ and each of them can appear as the first vertex of a Hamiltonian path of $G_1$. So, say, $x$ has all the vertices of $G_2$ as out-neighbors. But then the out-degree of $x$ in $G$ in $5$, a contradiction. In any case, we have $h(8)=4$.
	
	For $n=9$ we have $h(9) \ge h(8)=4$ but on the other hand, $R_9$ shows that $h(9) \le 4$ and therefore $h(9)=4$.
	
	For $n=10$, If the out-degree is one of $[9] \setminus \{3,6\}$ then we have $h(G) \ge 5$ by the recursive construction in the proof of Lemma \ref{l:h7}. So, we can assume that five vertices have out-degree $3$ and five vertices have out-degree $6$. Assume without loss of generality that $d^+(v)=6$. Assume for contradiction that $h(G)=4$.
	Then $G_1=G[N^{+}(v)]$ must be $R_6$ and $G_2=G[N^{-}(v)]$ must be $R_3=C_3$. But then the vertices $X_1$, the first triangle of $R_6$, must have all the vertices of $G_2$ as out-neighbors and hence their out-degree in $G$ is $7$, contradicting the assumption. Hence we have $h(10)=5$.
\end{proof} 

\begin{lemma}\label{l:lower}
	For all $n \ge 4$, $h(n) \ge (4n-10)/7$.
\end{lemma}
\begin{proof}
	We prove the lemma by induction on $n$.
	The assertion holds for all $4 \le n \le 10$ by Lemma \ref{l:h10}.
	Using Lemma \ref{l:h7} we also obtain that $h(11) \ge 5$, $h(12) \ge 6$, $h(13) \ge 6$ and $h(14) \ge 7$.
	For $n \ge 15$, we have by Lemma \ref{l:h7} that for some
	$\frac{1}{4}n-\frac{1}{2} \le x \le \frac{3}{4}n-\frac{1}{2}$ it holds that $h(n) \ge h(x)+h(n-1-x)+2$.
	Observe that $x \ge 4$ and $n-1-x \ge 4$. Hence by induction we obtain that
	$$
	h(n) \ge \frac{4x-10}{7} + \frac{4(n-1-x)-10}{7}+2 = \frac{4n-10}{7}\;.
	$$
\end{proof}
\noindent
Lemma \ref{l:lower} and the construction of $R_n$ together give Theorem \ref{t:1}. \qed

\section{Shortcuts and shortcut trees}

In this section we prove Theorem \ref{t:3}.
To prove the lower bound of Theorem \ref{t:3} we need the following lemma.
For a vertex $v$ of a tournament $G$, let $m(v)$ denote the number of $T_3$ of $G$ having $v$ as the middle vertex, namely $v$ is neither the source nor the sink of the $T_3$.
\begin{lemma}\label{l:mv}
	Let $G$ be a tournament with $n$ vertices. There exists a vertex $v$ with $m(v)$ at least:
	\begin{itemize}
		\item $(n-1)(n-3)/8$ if $n$ is odd.
		\item $\lceil (n-2)^2/8 \rceil$ if $n$ is even.
	\end{itemize}
\end{lemma}
\begin{proof}
	If $L$ denotes the number of $T_3$ in $G$, then
	$$
	\sum_{v \in V(G)} m(v) = L = \sum_{v \in V(G)} \binom{d^+(v)}{2}\;.
	$$
	Since $\binom{x}{2}$ is convex, the right hand side of the last equality is minimized when all
	$d^+(v)$ are as equal as possible. So, when $n$ is odd each is $(n-1)/2$ and when $n$ is even
	half are $n/2-1$ and half are $n/2$. Hence, when $n$ is odd, the last sum is always at least
	$n(n-1)(n-3)/8$ so the result holds by averaging and when $n$ is even the last sum is always at least $n(n-2)^2/8$ so the result holds by averaging.
\end{proof}

Define $z(n)$ recursively to be the following function of a positive integer $n$.
$z(1)=z(2)=0$.
If $n \ge 3$ then
$$
z(n) = \min_{x = 1,\ldots,\lfloor (n-1)/2 \rfloor} \left \lceil \frac{(n+7)(n-3)}{8} \right\rceil + z(x)+ z(n-x-1)\;.
$$
\begin{lemma}\label{l:hnzn}
	$t(n) \ge z(n)$ for all $n \ge 1$.
\end{lemma}
\begin{proof}
	We use induction on $n$. The cases $n=1,2,3$ hold trivially since $t(n)=z(n)=0$ in these cases.
	So we assume that $n \ge 4$ and that the lemma holds for all positive integers smaller than $n$.
	Suppose now that $G$ is a tournament with $n$ vertices.
	Let $v$ be a vertex with $m(v)$ at least as large as the value guaranteed by Lemma \ref{l:mv}
	and observe that $m(v) \ge 1$.
	Let $x=d^+(v)$ and consider the tournaments $G_1=G[N^+(v)]$ and $G_2=G[N^{-}(v)]$
	where $G_1$ has $x$ vertices and $G_2$ has $n-x-1$ vertices.
	Notice that $n-2 \ge x \ge 1$ since $m(v) \ge 1$.
	
	Let $P_1$ be a spanning shortcut tree of $G_1$ with at least $t(x)$ shortcuts and let $P_2$
	be a spanning shortcut tree of $G_2$ with at least $t(n-x-1)$ shortcuts.
	Consider now the binary tree $P$ where $v$ is the root, its left subtree is
	$P_2$ and its right subtree is $P_1$. So $P$ is a spanning shortcut tree of $G$.
	Now, all the shortcuts of $P_1$ and $P_2$ remain shortcuts of $P$. In addition, for each
	$T_3$ of $G$ of the form $u,v,w$ the edge $uw$ is a shortcut edge of $P$ and notice that there are
	$m(v)$ such shortcut edges. Finally, every edge of the form $uv$ where $u \in N^{-}(v)$
	is not the root of $P_2$ is a shortcut edge of $P$ and so is every edge of the form
	$vw$ where $w \in N^{+}(v)$ is not the root of $P_1$. Altogether, these are
	$n-3$ additional shortcut edges. So the number of shortcut edges of $P$ is at least $t(x)+t(n-x-1)+m(v)+n-3$. Thus, when $n$ is odd we have by the induction hypothesis, by the definition of $z(n)$ and by Lemma \ref{l:mv} that
	\begin{align*}
	& ~~~~ t(x)+t(n-x-1)+m(v)+n-3\\
	& \ge z(x)+z(n-x-1)+ \frac{(n-1)(n-3)}{8}+n-3\\
	&  = z(x)+z(n-x-1)+ \frac{(n+7)(n-3)}{8}\\
	& \ge z(n)\;.
	\end{align*}
	As the argument holds for an arbitrary tournament with $n$ vertices we have that
	$t(n) \ge z(n)$ when $n$ is odd. The same argument holds in the even case where we use
	the bound $\lceil (n-2)^2/8 \rceil$ of Lemma \ref{l:mv}.
\end{proof}

By Lemma \ref{l:hnzn}, to obtain a lower bound for $t(n)$ it suffices to obtain a lower bound for
$z(n)$.  This is established in the following lemma.
\begin{lemma}\label{l:z-lower}
	$z(n) \ge n^2/4+(n\ln n)/10 - 80$ for all $n \ge 1$.
	Similarly, $z(n) \ge n^2/4+(n\ln n)/5 - 3122$ for all $n \ge 1$.
\end{lemma}
\begin{proof}
	We prove the lemma by induction on $n$.
	We will prove the first bound. The second bound is proved similarly.
	As to the first bound, it is a simple computational task to verify that $z(n)$ satisfies the claimed inequality for all $n \le 600$. See Table \ref{table:1} for values of small $n$
	where $z(n)$ and $n^2/4+(n\ln n)/10 - 80$ are compared.
	Suppose now that $n > 600$ and that the lemma holds
	for all positive integers smaller than $n$.
	By the definition of $z(n)$, there exists an integer $1 \le x \le (n-1)/2$ such that
	$$
	z(n) \ge \frac{(n+7)(n-3)}{8}  + z(x)+ z(n-x-1)\;.
	$$
	So by the induction hypothesis,
	$$
	z(n) \ge \frac{(n+7)(n-3)}{8} + \frac{x^2}{4} + \frac{x}{10} \ln x - 80 + \frac{(n-x-1)^2}{4} + \frac{n-x-1}{10}\ln(n-x-1) - 80 \;.
	$$
	Viewing the right hand side as a real function of $x$ in $(0,(n-1)/2]$ it attains a minimum at
	$x=(n-1)/2$ and therefore,
	\begin{align*}
	z(n) & \ge \frac{(n+7)(n-3)}{8} + \frac{(n-1)^2}{8} + \frac{n-1}{10}\ln((n-1)/2)- 160\\
	& = \frac{n^2}{4}+ 0.15n-162.4 +\frac{n-1}{10}\ln(n-1)\\
	& \ge \frac{n^2}{4} + \frac{n\ln n}{10} - 80
	\end{align*}
	where we only need to show that the last inequality holds for all $n > 600$. Indeed, it amounts to show that for all
	$n > 600$ we have:
	$$
	1.5n \ge n\ln n - (n-1)\ln (n-1) + 824\;.
	$$
	But observe that for $n \ge 600$ we indeed have
	\begin{align}
	1.5 n & \ge 2+\ln(n-1)+824 \label{e:1}\\
	& \ge \ln \left[\left(1+\frac{1}{n-1}\right)^n \right] + \ln(n-1) + 824 \nonumber \\
	& = n\ln n - (n-1)\ln (n-1) + 824\;. \nonumber
	\end{align}
	It is clear from the proof that for any $k > 4$ we can obtain a bound of the form
	$z(n) \ge n^2/4+(n\ln n)/k - C$ for all $n \ge 1$ if we can find a suitable $C$ and a starting point
	$n_0$ for the inductive step (such as $C=80$ and $n_0=600$ in the case $k=10$ just proved).
	For example, in the case $k=5$ and the bound $z(n) \ge n^2/4+(n\ln n)/5 - C$, rewriting the proof
	amount to showing that there are $n_0$ and $C$ such that $z(n) \ge n^2/4+(n\ln n)/5 - C$ for
	all $n \le n_0$ and that for $n \ge n_0$ we have, analogous to (\ref{e:1}) that
	$$
	0.25n \ge  2+\ln(n-1) +5(C+2.5-0.2)\;.
	$$ 
	Indeed the last inequality holds for all $n \ge 62540$ and for $C=3122$
	and a simple computer verification shows that $z(n) \ge n^2/4+(n\ln n)/5 - 3122$ for all
	$n \le 62540$.
	The choice $C=3122$ is optimal in this case since $z(16383)=67129347$ and
	$\lceil (16383)^2/4+16383\ln (16383)/5 \rceil = 67132469$.
\end{proof}

\begin{table}
	\centering
	\begin{tabular}{|c|c|c|c|}
		\hline
		$n$ & $z(n)$ & $\lceil n^2/4+(n\ln n)/10 - 80 \rceil$ & gap\\
		\hline
		$ \le 17$ & $\ge 0$ & $< 0$ & $> 0$\\
		\hline
		$18$ & $ 74$ & $7$ & 67\\
		\hline
	$50$ & $618$ & $565$ &  $53$\\
	\hline
	$100$ & $2508$ & $2467$ &  $41$\\
	\hline
	$150$ & $5657$ & $5621$ &  $36$\\
	\hline
	$200$ & $10062$ & $10026$ &  $36$\\
	\hline
	$250$ & $15696$ & $15684$ &  $12$\\
	\hline
	$300$ & $22635$ & $22592$ &  $43$\\
	\hline
	$350$ & $30805$ & $30751$ &  $54$\\
	\hline
	$400$ & $40219$ & $40160$ &  $59$\\
	\hline
	$450$ & $50874$ & $50820$ &  $54$\\
	\hline
	$500$ & $62765$ & $62731$ &  $34$\\
	\hline
	$550$ & $75965$ & $75893$ &  $72$\\
	\hline
	$600$ & $90415$ & $90304$ &  $111$\\
	\hline
	\end{tabular}
	\caption{Some small values of $z(n)$ compared to $n^2/4+(n\ln n)/10 - 80$.}
	\label{table:1}
\end{table}

\begin{lemma}\label{l:algo-t}
There is an $O(n^2)$ time algorithm that finds a spanning shortcut tree with at least $\frac{1}{2}\binom{n}{2} +(n\ln n)/5 - 3122$ edges.
\end{lemma}
\begin{proof}
	By Lemmas \ref{l:hnzn} and \ref{l:z-lower}, we only need to show how to locate a vertex $v$ in a tournament for which $m(v)$ is maximum in $O(n^2)$ time, as we can then recursively solve the problem on the sub-tournaments $G[N^+(v)]$ and $G[N^-(v)]$ which each have $\Theta(n)$ vertices.
	(Since when $m(v)$ is maximum, it is quadratic in $n$ so the in-degree and out-degree of $v$
	are $\Theta(n)$ each.) Indeed we can compute $m(v)$ for each vertex $v$ in $O(n)$ time
	given all the in-degrees and out-degrees of all other vertices. The number of $T_3$ in which $v$
	is sink is $\binom{d^-(v)}{2}$. The number of $T_3$ in which $v$ is not a source is
	$\sum_{u \in N^-(v)} (d^+(u)-1)$. So
	$$
	m(v) = \left(\sum_{u \in N^-(v)} (d^+(u)-1)\right) - \binom{d^-(v)}{2}\;.
	$$
\end{proof}

To prove the upper bound of Theorem \ref{t:3} we consider random tournaments.
We first recall them and establish some properties that hold in them with positive probability.
A random tournament is the uniform probability space $G(n)$ of all tournaments labeled on vertex set $[n]$.
So $G \sim G(n)$ is generated by choosing for each pair $\{i,j\}$ the direction of the edge connecting them by a fair coin flip, and all $\binom{n}{2}$ choices are independent.

We shall require the following definition. For a tournament $G$ with vertex set $[n]$, for a given
vector $K \in [n]^k$ of $k$ distinct vertices and for a vector $D \in \{+,-\}^k$, the {\em agreement set} $A(K,D)$ is
$$
A(K,D) = \{v \,|\, v \in [n] \setminus K \textrm{ such that } v \in N^{D(i)}(K(i)) \textrm{ for } i=1,\ldots,k\}\;.
$$
Namely, each vertex $v$ of $A(K,D)$ has the property that if for $x=K(i)$ the corresponding $D(i)$ is $+$ then
$xv \in E(G)$ and if the corresponding $D(i)$ is $-$ then $vx \in E(G)$.
For convenience, define also $A(\emptyset,\emptyset)=[n]$.
\begin{lemma}\label{l:random-tour}
	Let $G \sim G(n)$. For $n$ sufficiently large, with positive probability the following hold:
	\begin{enumerate}
		\item
		$G$ has no transitive sub-tournament on more than $3\log n$ vertices.
		\item
		For all $k \ge 0$, for all vectors $K \in [n]^k$ of $k$ distinct vertices such that $G[K]$ is transitive and for all vectors $D \in \{+,-\}^k$ the following holds.
		If $|A(K,D)| = r$ then for any vertex $v \in A(K,D)$, the number of $T_3$ in $G[A(K,D)]$ in which $v$ is the middle vertex is at most $a(r-1-a)/2+r\log n$ where $a$ is the in-degree of $v$ in $G[A(K,D)]$.
	\end{enumerate}
\end{lemma}
\begin{proof}
	It is well-known that the maximum transitive sub-tournament of $G \sim G(n)$ almost surely has
	at most $2\log n(1+o(1))$ vertices \cite{EM-1964}. In particular, with probability at least $0.9$
	it does not have a transitive sub-tournament  with more than $3\log n$ vertices.
	Notice that this also implies that the {\em total} number of transitive sub-tournaments (of any size) 
	in $G$ is at most $n^{3\log n}$.
	
	Next, suppose that $R \subseteq [n]$ is a given set of $r$ vertices,
	and consider $G[R]$ which is a random tournament on $r$ vertices.
	For $v \in R$ let $m(v)$ denote the number of $T_3$ in $G[R]$ in which $v$ is the middle vertex.
	Suppose we are {\em given} the information of the set $A \subset R$ of in-neighbors of $v$
	in $G[R]$ and let $a=|A|$. We compute the probability that $m(v)$ is larger than $a(r-1-a)/2+r \log n$ given that information. Notice that given $A$ we also know $B$, the set of out-neighbors of $v$ in $G[R]$.
	For each $u \in A$ and $w \in B$ let $X_{u,w}$ be the indicator random variable for the event that the
	triple $u,v,w$ induces a $T_3$ in $G[R]$ in which $v$ is the middle vertex. Then
	$m(v)$ is just the sum of these indicator variables and distributed $Bin(a(r-1-a),\frac{1}{2})$.
	The probability that $m(v)$ is larger than its expected value $a(r-1-a)/2$ by more than
	$r \log n$ is at most
	$$
	\exp\left(-\frac{2(r\log n)^2}{a(r-1-a)}\right) < \exp\left(-8(\log n)^2\right)< \frac{1}{n^{5\log n}}
	$$
	where we have used the Chernoff's large deviation inequality (see \cite{AS-2004})
	and that $a(r-1-a) < r^2/4$.
	As this upper bound $n^{-5\log n}$ for the probability holds regardless of the
	given set $A$ of the in-neighbors of $v$ in $G[R]$, it follows that
	\begin{equation}\label{e:mv}
		\Pr\left[m(v) \ge a(r-1-a)/2+ r\log n\right] < \frac{1}{n^{5\log n}}\;.
	\end{equation}
	
	Now, suppose we are told that $R=A(K,D)$ for given vectors $K$ and $D$. Observe that given this information, $G[R]$ is still completely random as the $\binom{r}{2}$ coin flips determining $G[R]$ are
	independent of this information. Hence, for given $K,D$ such that $|A(K,D)| = r$,
	(\ref{e:mv}) holds.
	
	Now how many pairs $(K,D)$ are there such that $G[K]$ is transitive?
	Recall that with probability at least $0.9$, there are only at most $n^{3\log n}$ transitive sub-tournaments
	and that they are all of size at most $3\log n$. So, in this case there are only at most $2^{3 \log n}n^{3\log n} < n^{4 \log n}$ such pairs.
	Hence we have by the union bound and by (\ref{e:mv}) that {\em all} the corresponding $A(K,D)$
	and all their vertices $v$ satisfy $m(v) \le a(r-1-a)/2+ r\log n$ where $r=|A(K,D)|$, $a$ is the in-degree of $v$ in $G[A(K,D)]$,
	and this occurs with probability at least
	$$
	0.9 - n^{4 \log n} \cdot n \cdot  \frac{1}{n^{5\log n}} > 0\;.
	$$
\end{proof}

We now fix a tournament $G$ on vertex set $[n]$ for which the two properties stated in Lemma \ref{l:random-tour} hold.
\begin{lemma}\label{l:upper}
	Any spanning shortcut tree of $G$ has at most $\frac{1}{2}\binom{n}{2} +4n\log^{2} n$ shortcut edges.
\end{lemma}
\begin{proof}
	Fix some spanning shortcut tree $B$ of $G$, rooted at some vertex $x$. We say that
	a vertex $u$ is at {\em level $\ell$} of $B$ if its tree distance from $x$ is $\ell$.
	Let $L_\ell$ denote all the vertices at level $\ell$ (so $L_0 = \{x\}$).
	For any vertex $v$, let $B_v$ denote the sub-tree of $B$ rooted at $v$.
	
	First we claim that the level of each vertex is at most $3\log n$.
	Indeed, for any vertex $v$ at level $\ell$, its set of ancestors up to the root $x$
	is a set of $\ell$ vertices which induces a $T_\ell$ in $G$, so by the first property of $G$, we have that
	$\ell \le 3\log n$. Let the maximum level be denoted by $k$, so $k \le 3\log n$ and
	$\cup_{\ell=0}^k L_\ell = [n]$.
	
	Let $F$ be the set of edges of $G$ that connect a vertex with its ancestor. As the maximum level is $k$
	we have that $|F| \le k n \le 3n \log  n$.
	For an edge $uw$ of $G$, we say that it is {\em separated} at level $\ell$ by $v$ if the lowest common ancestor of $u$ and $w$ is $v$ where $v \neq u,w$ and $v$ is in level $\ell$.
	For example, in Figure \ref{f:shortcut-tree}, the edge of $G$ connecting $d$ and $e$ is separated at level $1$ by  $b$ (since $b$ is their lowest common ancestor and $b$ is in level $1$) and the  edge of $G$ connecting $b$ and $f$ is separated at level $0$ by $a$.
	Observe that the notion of separation is well-defined for all edges in $E(G) \setminus F$.
	Let $E_\ell$ be the set of edges that are separated at level $\ell$. So we have
	$\cup_{\ell=0}^{k-1} E_\ell = E(G) \setminus F$.
		
	Next we estimate, for each level $\ell$, the number of shortcut edges in $E_\ell$.
	Let therefore $v \in L_\ell$. Let $K$ be the set of ancestors of $v$ up to the root.
	So $K$ induces a transitive sub-tournament on $\ell$ vertices and
	the vertex set of $B_v$ is just the agreement set $A(K,D)$ for some vector $D \in \{+,-\}^{\ell}$ (notice that if $v=x$ is the root then $K=\emptyset$ and $B_x=[n]$ in this case).
	
	Observe that the number of edges separated by $v$ that are shortcut edges is precisely $m(v)$,
	where as before $m(v)$ denotes the number of $T_3$ in $G[A(K,D)]$ in which $v$ is the middle vertex.
	Now, by the second property of $G$ we have that
	$m(v) \le a(r-1-a)/2+ r \log n$ where $r=|A(K,D)|=|B_v|$ and $a$ is the in-degree of $v$ in $G[A(K,D)]$.
	Observe that $a(r-1-a)$ is also the number of edges separated by $v$.
	So,  summing over all vertices $v \in L_\ell$,  the overall number of shortcut edges in $E_\ell$ is
	at most
	$$
	|E_\ell|/2 + n \log n\;.
	$$
	Summing this over all levels and adding also the edges of $F$ (most of which are shortcuts) we obtain that
	the number of shortcut edges of $B$ is at most
	$$
	|F| + \sum_{\ell=0}^{k-1} \left(\frac{|E_\ell|}{2} + n \log n\right) \le 3n \log n + \frac{|E(G)|}{2} + 3n \log^2 n \le \frac{1}{2}\binom{n}{2}+4n \log^2 n\;.
	$$
\end{proof}

Finally, lemmas \ref{l:hnzn}, \ref{l:z-lower}, \ref{l:algo-t}, and \ref{l:upper} together prove Theorem \ref{t:3}. \qed

\section{Hop-complete paths}

In this section we prove Theorem \ref{t:2}.
Let $P_s^{(2)}$ denote the second power of a directed path on $s$ vertices, or equivalently, a hop-complete path of order $s$.
We prove that every tournament with $n$
vertices contains a $P_s^{(2)}$ with $s \ge n^{0.295}$.

For two distinct vertices $u,v$ of a tournament, let $N^{+}(u,v)$ denote the set of common out-neighbors of $u,v$
and let $N^{-}(u,v)$ denote the set of common in-neighbors of $u,v$. Put $d^{+}(u,v) = |N^{+}(u,v)|$
and $d^{-}(u,v) = |N^{-}(u,v)|$. For a tournament $G$, let
$$
\Delta_2(G)=\max_{u,v \in V(G)} \min \{d^{+}(u,v), d^{-}(u,v)\}\;.
$$
It is not difficult to construct tournaments where already the maximum of $\min\{d^{+}(u), d^{-}(u)\}$ is
less than $n/4$. The following lemma proves that if we consider common neighborhoods of pairs of vertices, we can still guarantee that $\Delta_2(G)$ is a (not so small) fraction of $n$.
\begin{lemma}\label{l:Delta}
	Let $G$ be a tournament with $n$ vertices, then $\Delta_2(G) \ge n\left(\frac{3-\sqrt{5}}{8}\right) - 4$.
\end{lemma}
\begin{proof}
	The lemma trivially holds if $n \le 40$ so assume that $n > 40$.
	By Lemma \ref{l:mv}, $G$ has a vertex $v$ with $m(v) \ge (n-1)(n-3)/8$, namely $v$ is the middle vertex of at least $(n-1)(n-3)/8$ copies of $T_3$ in $G$. Let $A=N^+(v)$ and $B=N^-(v)$. Since every $T_3$
	in which $v$ is a middle vertex contains a source vertex from $B$ and a sink vertex from $A$, there are
	at least $(n-1)(n-3)/8$ edges directed from $B$ to $A$, that is, $|E(B,A)| \ge (n-1)(n-3)/8$.
	In particular, each of $A$ and $B$ is of size linear in $n$. Assume without loss of generality that
	$|A| \ge |B|$ and let $|A|=\gamma(n-1)$ and $|B|=(1-\gamma)(n-1)$ where $\frac{1}{2} \le \gamma < 1$.
	In fact, we can first show that $\gamma \le 0.863$. Indeed,
	$$
	(n-1)^2\gamma(1-\gamma) = |A||B| \ge |E(B,A)| \ge \frac{(n-1)(n-3)}{8}
	$$
	implies that
	$$
	\gamma^2-\gamma + \frac{1}{8}-\frac{1}{4(n-1)} \le 0
	$$
	so
	$$
	\gamma \le \frac{1}{2}+\sqrt{\frac{1}{8}+\frac{1}{4(n-1)}} < 0.863
	$$
	where we have used that $n > 40$.
	Let $\beta$ be a parameter to be set later where $0 \le \beta < 1-\gamma$.
	Let $X \subseteq A$ be the set of vertices that have more than $\beta(n-1)$ in-neighbors in
	$B$ and let $\alpha=|X|/(n-1)$. Since $|E(B,A)| \ge (n-1)(n-3)/8$ we have that
	$$
	\alpha(n-1)(1-\gamma)(n-1) + (\gamma - \alpha)(n-1)\beta(n-1) \ge \frac{(n-1)(n-3)}{8}\;.
	$$
	Equivalently,
	$$
	\alpha(1-\gamma) + (\gamma - \alpha)\beta \ge \frac{1}{8} - \frac{1}{4(n-1)}
	$$
	implying that
	$$
	\alpha \ge \frac{1-8\beta\gamma}{8(1-\gamma-\beta)} - \frac{1}{4(n-1)(1-\gamma-\beta)}\;.
	$$
	Since every tournament with $x$ vertices has a vertex with out-degree at least $(x-1)/2$,
	there is a vertex $u \in X$ such that the out-degree of $u$ in $G[X]$ is at least
	$(\alpha(n-1)-1)/2$. So, there are at least $(\alpha(n-1)-1)/2$ vertices in $N^+(u,v)$ and
	at least $\beta(n-1)$ vertices in $N^-(u,v)$.
	It follows that
	\begin{equation}\label{e:min}
	\Delta_2(G) \ge (n-1) \cdot \min \left\{ \beta\,,\, \frac{1-8\beta\gamma}{16(1-\gamma-\beta)}\right\} -   
	\frac{1}{8(1-\gamma-\beta)}-\frac{1}{2}\;.
	\end{equation}
	Consider first the case $0.8 \le \gamma \le 0.863$. In this case, we will use $\beta=0.1$ (this choice satisfies $\beta < 1-\gamma$).
	Observe that $0.1 \le (1-0.8\gamma)/(16(0.9-\gamma))$  for $\gamma$ in this range, so we obtain
	$$
	\Delta_2(G) \ge 0.1(n-1)-\frac{1}{8(0.9-\gamma)}-\frac{1}{2} \ge n\left(\frac{3-\sqrt{5}}{8}\right) - 4\;.
	$$
	Consider next the case $\frac{1}{2} \le \gamma \le 0.8$. For $\gamma$ in this range, the function $f(\beta)=(1-8\beta\gamma)/(16(1-\gamma-\beta))$ is a monotone decreasing function of $\beta$. Thus, the optimal choice of $\beta$ maximizing the minimum term in (\ref{e:min}) is obtained when $\beta=f(\beta)$ which, in turn, occurs for
	$$
	\beta= \frac{1}{2}-\frac{\gamma}{4}-\frac{\sqrt{3+\gamma^2-4\gamma}}{4}\;.
	$$
	So, choosing this value of $\beta$ and plugging it in (\ref{e:min}) we obtain that
	$$
	\Delta_2(G) \ge (n-1)\left(\frac{1}{2}-\frac{\gamma}{4}-\frac{\sqrt{3+\gamma^2-4\gamma}}{4}\right)
	-\frac{1}{8(1/2-3\gamma/4+\sqrt{3+\gamma^2-4\gamma}/4)}-\frac{1}{2}\;.
	$$
	The term multiplying $(n-1)$ in the last expression is minimized when $\gamma=\frac{1}{2}$ in which case it equals
	$\frac{3-\sqrt{5}}{8}$. The expression $8(1/2-3\gamma/4+\sqrt{3+\gamma^2-4\gamma}/4)$ is minimized at $\gamma=0.8$
	in which case it is still larger than $0.52$. Hence,
	$$
	\Delta_2(G) \ge (n-1)\left(\frac{3-\sqrt{5}}{8}\right)-\frac{1}{0.52}-\frac{1}{2} \ge n\left(\frac{3-\sqrt{5}}{8}\right) - 4\;.
	$$
\end{proof}
We note that it is possible to slightly improve the constant $\frac{3-\sqrt{5}}{8}$ in Lemma \ref{l:Delta} at the expense of a more involved proof since the lower bound for $m(v)$ in Lemma \ref{l:mv}
is only obtained for tournaments that are almost regular, and we did not assume this in the proof of Lemma \ref{l:Delta}. However, the improvement is rather small, so we omit the details.

Using Lemma \ref{l:Delta}, we can now prove Theorem \ref{t:2} by induction on $n$.
For all $n \le 1000$ the theorem clearly holds since $n^{0.295}$ is smaller than the cardinality
of the largest transitive tournament. Indeed, this is straightforward to verify since
every tournament with $2^{r-1}$ vertices contains $T_r$ \cite{stearns-1959} (so, say, for $512 \le n \le 1000$ we have
$n^{0.295} <  8$ while the tournament contains a $T_9$ so also a $P_9^{(2)}$).
Assume that $n > 1000$ and that the lemma holds for values smaller than $n$.
Let $G$ be a tournament with $n$ vertices.
Let $u,v$ be two distinct vertices of $G$ with
$d^{+}(u,v) \ge n\left(\frac{3-\sqrt{5}}{8}\right) - 4$ and  $d^{-}(u,v) \ge n\left(\frac{3-\sqrt{5}}{8}\right) - 4$. By Lemma \ref{l:Delta}, such a pair of vertices exists.

Let $G_1=G[N^{-}(u,v)]$ and $G_2=G[N^{+}(u,v)]$.
By the induction hypothesis, $G_1$ has a $P_{s_1}^{(2)}$ with $s_1 \ge (d^{-}(u,v))^{0.295}$
and $G_2$ has a $P_{s_2}^{(2)}$ with $s_2 \ge (d^{+}(u,v))^{0.295}$.
Assume without loss of generality that $uv \in E(G)$.
Construct a path by concatenating the $P_{s_1}^{(2)}$ of $G_1$ followed by $u$, followed by $v$,
followed by the $P_{s_2}^{(2)}$ of $G_2$. Then this concatenated path has $s_1+s_2+2$ vertices and is a
$P_{s_1+s_2+2}^{(2)}$. Its order satisfies
$$
s_1+s_2+2 \ge  2\left(n\left(\frac{3-\sqrt{5}}{8}\right) - 4\right)^{0.295}+2 \ge n^{0.295}
$$
where in the last inequality we have used $n \ge 1000$ and that
$$
\frac{\ln 2}{\ln(8/(3-\sqrt{5}))} > 0.295\;.
$$
\qed

We end this section by showing that there are tournaments
of order $n$ where the longest $k$th power of a path has order at most $nk/2^{k/2}$.
\begin{proposition}\label{p:1}
	For every $k \ge 2$ there are infinitely many integers $n$ for which there is a tournament with
	$n$ vertices where the longest $k$'th power of a path has at most $nk/2^{k/2}$ vertices.
\end{proposition}
\begin{proof}
	Let $g(k)$ denote the largest integer such that there is a tournament
on $g(k)$ vertices which does not contain $T_{k+1}$. So, for example $g(2)=3$, $g(3)=7$,
$g(4)=13$ \cite{RP-1970} and Erd\H{o}s and Moser \cite{EM-1964} proved that
$g(k) \ge 2^{k/2}$. Let $F_k$ denote some tournament on $g(k)$ vertices with no $T_{k+1}$.	
Suppose that $n$ is a multiple of $g(k)$ and define the tournament
$R(n,k)$ as follows. Take $n/g(k)$ pairwise vertex-disjoint copies of $F_k$
denoted by $X_1,\ldots,X_{n/g(k)}$.
Now, if $i < j$, orient all $g(k)^2$ edges connecting $X_i$ and $X_j$ from $X_i$ to $X_j$.
Observe that $R(n,2)$ is just the tournament $R_n$ from Section 2.
Notice that in every directed path of $R(n,k)$ if a vertex $v \in X_i$ is on the path and a vertex
$u \in X_j$ is somewhere after $v$ on the path then $i \le j$.
So all the vertices of the path that belong to some $X_i$ must be consecutive on the path.
But since $R(n,k)[X_i]$ is isomorphic to $F_k$ and since $F_k$ does not contain $T_{k+1}$,
any $k$'th power of a path in $R(n,k)$ contains at most $k$ vertices of $X_i$.
It follows that the longest $k$th power of a path in $R(n,k)$ is of order at most
$nk/g(k) \le nk/2^{k/2}$.
\end{proof}

\section*{Acknowledgment}

I thank the referee for very useful comments and suggestions.

\bibliographystyle{plain}

\bibliography{references}

\end{document}